\documentclass[a4paper]{article}

\usepackage{amsmath,amsthm}
\usepackage{amssymb,latexsym}
\usepackage{enumerate}
\usepackage[english]{babel}

\theoremstyle{definition}

\newtheorem{rem}{Remark}
\newtheorem*{xrem}{Remark}
\theoremstyle{plain}
\newtheorem{corollary}{Corollary}
\newtheorem{proposition}{Proposition}

\newtheorem{ex}{Example}
\newtheorem{theorem}{Theorem}

\newcommand{\sr}[1]{\mathfrak{M}_{#1}}
\newcommand{\norma}[1]{\left\| #1 \right\| }
\newcommand{\abs}[1]{\left| #1 \right| }
\newcommand{\R}{\mathbb{R} }
\newcommand\A[1]{\pmb{#1}}

\title{When is a family of generalized means a scale?}
\author{Pawe\l{} Pasteczka\\
University of Warsaw\\
E-mail: ppasteczka@mimuw.edu.pl}
\date {December 15, 2011}

\begin{document}

\maketitle

\begin{abstract}
For a family $\{k_t \,\vert \,t \in I\}$ of real $\mathcal{C}^2$ 
functions defined on $U$ ($I$, $U$ -- open intervals) and satisfying 
some mild regularity conditions, we prove that the mapping 
$I \ni t \mapsto k_t^{-1}\bigl(\sum_{i=1}^n w_i k_t(a_i)\bigr)$ 
is a continuous bijection between $I$ and $(\min\underline{a}, 
\,\max\underline{a})$, for every fixed non-constant sequence 
$\underline{a} = \bigl(a_i\bigr)_{i=1}^n$ with values in $U$ 
and every set, of the same cardinality, of positive weights 
$\underline{w} = \bigl(w_i\bigr)_{i=1}^n$. In such a situation 
one says that the family of functions $\{k_t\}$ generates 
a {\it scale\,} on $U$. The precise assumptions in our result 
read (all indicated derivatives are with respect to $x \in U$)
\begin{enumerate}[\upshape (i)]
\item $k'_t$ does not vanish anywhere in $U$ for every $t \in I$,
\item $I \ni t \mapsto \frac{k''_t(x)}{k'_t(x)}$ is increasing, 
1--1 on a dense subset of $U$ and onto the image $\R$ 
for every $x \in U$. 
\end{enumerate}
\noindent This result makes possible three things. 1)\,a new and 
extremely short proof of the classical fact that {\it power means\,} 
generate a scale on $(0,+\infty)$, 2)\,a short proof of a fact, which 
is in a direct relation to two results established by Koles\'arov\'a 
in 2001, that, for every strictly increasing convex and $\mathcal{C}^2$ 
function $g \colon (0,\,1) \to (0,\,+\infty)$, the class 
$\{\sr{g_\alpha}\}_{\alpha \in (0,\,+\infty)}$ of quasi-arithmetic 
means (see Introduction for the definition) generated by functions 
$g_\alpha$, $g_\alpha(x) = g(x^\alpha)$, $\alpha \in (0,\,+\infty)$, 
generates a scale on $(0,1)$ between the geometric mean and maximum 
(meaning that, for every $\underline{a}$, $\underline{w}$, if 
$s \in \bigl(\prod_{i = 1}^n a_i^{\,w_i},\,\max(\underline{a})\bigr)$ 
then there exists exactly one $\alpha$ such that 

\noindent $\sr{g_\alpha}(\underline{a},\underline{w}) = s$). 

3)\,a brief proof of one of the classical results of the Italian 
statistics' school from the 1910-20s that the so-called {\it radical 
means\,} generate a scale on $(0,\, +\infty)$. 
\end{abstract}

\section{Introduction}
One of the most popular families of means encountered in the literature 
consists of quasi-arithmetic means. That mean is defined for any continuous 
strictly monotone function $f \colon U \to \R$, $U$ -- an open 
interval. When $\underline{a} = (a_1,\dots,\,a_n)$ is a sequence of 
points in $U$ and $\underline{w} = (w_1,\dots,\,w_n)$ is a sequence 
of {\it weights\,} ($w_i > 0$, $w_1 + \cdots + w_n = 1$), then the 
mean $\sr{} = \sr{f}(\underline{a},\,\underline{w})$ is well-defined 
by the equality 
$$
f(\sr{}) = \sum_{i=1}^n w_i f(a_i)\,.
$$ 
According to \cite[pp.\,158--159]{hlp}, this family of means was dealt 
with for the first time in the papers \cite{deFinetti,kolmogoroff,nagumo} 
in the early thirties of the last century as a natural generalization 
of the power means. Clearly, it is also discussed in the by-now-classical 
encyclopaedic publications \cite{bullen,bvm}. One gets this family, 
containing the most popular means: arithmetic, geometric, quadratic, 
harmonic, by putting 
$$
f_r(x) = \begin{cases} 
x^r & \textrm{if\ } r \ne 0\\
\ln x& \textrm{if\ } r = 0
\end{cases}, 
$$ 
$x \in U = (0, \,+\infty)$, \,$r \in I = \R$. 
\vskip.5mm
We pass now to the notion of scale in the theory of means. 
If a non-constant vector $\underline{a} \in U^n$ and weights $\underline{w}$ 
are fixed then the mapping $f \mapsto \sr{f}(\underline{a},\,\underline{w})$ 
takes continuous monotone functions $f \colon U \to \R$ to the 
interval $(\min\underline{a}, \max\underline{a})$. One is interested in 
finding such families of functions $\{f_i \colon U \to \R\}_{i \in I}$, 
where $I$ is an interval, that for every non-constant vector $\underline a$ 
with values in $U$ and arbitrary fixed corresponding weights $\underline w$, 
the mapping $I \ni i \mapsto \sr{f_i}(\underline{a},\,\underline{w})$ be
a {\it bijection\,} onto $(\min\underline{a},\,\max\underline{a})$. Every 
such a family of means $\sr{f_i}$ is called {\it scale on $U$}. 

The problem of finding conditions, for a family of means, equivalent to 
its being a scale has been discussed for various families. For instance, 
a set of conditions pertinent for Gini means was presented in 
\cite{petrarcicbessack86}. Many results concerning means may be 
expressed in a compact way in terms of scales. Probably the most famous 
is the fact that the family of power means is a scale on $(0,\,+\infty)$. 
It was proved for the first time (for arbitrary weights) in \cite{besso}. 
More about the underlying history, as well as another proof, was given 
in \cite[p.\,203]{bullen}. In the last section of the present note we 
will present a new, extremely short proof of this classical fact.
\section{Comparison of means}
Dealing with means, we would like to know whether (a)\,one mean is not smaller 
than the other, whenever both are defined on the same interval and computed 
on same, but arbitrary, set of arguments. And, when (a) holds true, whether 
(b)\,the two means, evaluated on arguments, are equal only when all components 
in an input $\underline{a}$ are the same: $a_1 = a_2 = \dots = a_n$.
With (a) and (b) holding true, we would say that the first mean is 
{\it greater\,} than the second. 

As long as quasi-arithmetic means are concerned, the comparability 
of $\sr{f}$ and $\sr{g}$ as such turns out to be intimately related 
to the convexity of the function $f \circ g^{-1}$, see items (ii) 
and (iii) in Proposition~\ref{prop:basiccompare} below.

Unfortunately, however, when it comes to scales, the family of objects 
to handle becomes uncountable. Hence one is forced to use another tool, 
allowing to tell something about an uncountable family of means. 
Its concept goes back to a seminal paper \cite{mikusinski}. 
A key operator $A$ from \cite{mikusinski} (recalled below) 
is used in item (i) in our technically crucial 
Proposition~\ref{prop:basiccompare}.

In fact, let $U$ be an interval, $\mathcal{C}^{2\ne}(U)$ be the class of functions 
from $\mathcal{C}^{2}(U)$ with the first derivative vanishing nowhere in $U$. 
Within this class one defines $A \colon \mathcal{C}^{2\ne}(U) \rightarrow 
\mathcal{C}(U)$ by the formula
$$
A(f) = \frac{f''}{f'}\,.
$$
However, the operator $A$ will be used so often as to adopt the convention that,
for $a,\,b,\,c,\,\dots \in \mathcal{C}^{2\ne}(U)$, \,$\A{a},\,\A{b},\,\A{c},\,\dots$ 
\,stand for $A(a),\,A(b),\,A(c),\,\dots$ \,Due to \cite{mikusinski}, this operator has 
wide applications in the comparison of means -- see Proposition~\ref{prop:basiccompare}. 
In fact, it will enable us to compare means in huge families, not only in pairs. 
Precisely this kind of comparison was being advanced by Polish mathematicians in the late 1940s.

One of the most important facts was discovered by Mikusi\'nski, who published 
his result, \cite{mikusinski}, in "Studia Mathematica"\footnote{the flagship journal 
of the pre-war Lvov Mathematical School, established by H.\,Steinhaus and S.\,Banach.}.
It is quite surprising that such a useful result has not been included in the 
referential book \cite{bullen}.

We present both necessary and sufficient conditions, for a family of 
functions $\{k_t\}_{t \in I}$ defined on a common interval $U$, to generate
a scale on $U$. The key conditions in our Theorems~\ref{thm:mainresultL} 
and \ref{thm:mainresultR} are given in terms of the operator $A$. 
Reiterating, it is handy to compare means with its help. 
We begin with
\begin{theorem}
\label{thm:mainresultL}
Let $U$ be an interval, $I=(a,b)$ an open interval, $(k_{\alpha})_{\alpha \in I}$, 
$k_\alpha \in \mathcal{C}^{2\ne}(U)$ for all $\alpha$. 

If $I \ni \alpha \mapsto A(k_\alpha) (x) \in \R$ is increasing and 1--1 
on a dense subset of \,$U$, and is onto for all \,$x \in U$, then 
$(\sr{k_{\alpha}})_{\alpha \in I}$ is an increasing scale on $U$.
\end{theorem}

A proof of this theorem is given in Section~\ref{sec:MainResults}. As a matter 
of fact, we will need a wider version of the above theorem. Namely, we extend 
the setup as follows.

In the definition of a scale (see Introduction) one may replace $\min \underline{a}$ 
and $\max \underline{a}$ by arbitrary bounds $L(\underline{a},\,\underline{w})$ 
and $H(\underline{a},\,\underline{w})$ respectively, with some functions $L$ and 
$H$.\footnote{We slightly abuse the notation here, as most of the researchers 
active in the field of means do, e.\,g., in \cite[p.\,61]{bullen}} 
Then such a modified family of means is called a {\it scale between $L$ and $H$}.
Such generalization is very natural and is frequently used, e.\,g. in \cite[pp.\,323,\,364]{bullen}.

Bounds in a scale, in most cases, are either quasi-arithmetic means or $\min$, or $\max$.
In order to make the notation more homogeneous, we introduce two extra symbols $\perp$ 
and $\top$, and write henceforth, purely formally, $\sr{\perp} = \min$ and $\sr{\top} 
= \max$. We also adopt the convention that $A(\perp) = -\infty$ and $A(\top) = +\infty$. 

\vskip1mm
\textbf{Attention.} In some papers scales may as well be decreasing. In fact, 
we do not lose generality if we assume that all scales are increasing, because 
whenever a family $\{k_\alpha\}_{\alpha \in I}$ generates a decreasing scale and 
$\varphi \colon J \rightarrow I$ is continuous, decreasing, 1--1 and onto, then 
the family $\{k_{\varphi(\alpha)}\}_{\alpha \in J}$ generates an increasing scale 
(see, e.\,g., Proposition~\ref{prop:RadicalMeans} in Section~\ref{sec:applications}).

\begin{corollary}[Bounded Scale]
\label{col:borderscale}
Let $l,\,h \in \mathcal{C}^{2\ne}(U) \cup \{\perp,\,\top\}$. Let $U$ and $I = (a,b)$ 
be open intervals, $(k_{\alpha})_{\alpha \in I}$ be a family of functions, 
$k_\alpha \in \mathcal{C}^{2\ne}(U)$ for all $\alpha$. 

If $I \ni \alpha \mapsto A(k_\alpha) (x) \in \R$ is increasing (decreasing), 
1--1 on a dense subset of $U$ and onto $(A(l)(x),\,A(h)(x))$ for all $x \in U$, 
then $(\sr{k_{\alpha}})_{\alpha \in I}$ is an increasing (decreasing) scale 
between $\sr{l}$ and $\sr{h}$.
\end{corollary}

The proof is but a specification of the proof of Theorem~\ref{thm:mainresultL}. 

\begin{xrem}
If, in the above corollary, $l,\,h \in \mathcal{C}^{2\ne}(U)$, then it is enough 
to assume that the mapping $\alpha \mapsto A(k_\alpha) (x)$ be onto for almost 
all $x \in U$. (Then, by Theorem~\ref{thm:rightright}, one gets the convergence 
in $L_1$).
\end{xrem}

The strength of Theorem~\ref{thm:mainresultL} is visible in the following 
exercise. 
\begin{ex}
\label{ex:xax}
Let $U = (\frac{1}{e},\,+\infty)$ and $k_\alpha(x) = x^{\alpha x}$ 
for $\alpha \in \R \,\backslash \,\{0\}$.\\
Find a function $k_0$ such that the completed family 
$(k_\alpha)_{\alpha \in \R}$ generates 
a scale on \,$U$. 
\end{ex}

By the definition of the operator $A$, for $\alpha \ne 0$ there holds
$$
\A{k_\alpha}(x) = \frac{1}{x(\ln x + 1)} + \alpha (\ln x + 1)\,.
$$
In view of Theorem~\ref{thm:mainresultL} we will be done, provided
$\alpha \mapsto \A{k_\alpha}(x)$ is increasing, 1--1 and onto 
$\R$ for all $x \in U$. But 
$$
\R \,\backslash \,\{0\} \ni \alpha \mapsto \A{k_\alpha} (x) \in \R
\,\backslash \,\left\{\frac{1}{x(\ln x + 1)}\right\}\quad \textrm{for all\ }x \in U\,.
$$ 
Hence it is natural to take $k_0 = A^{-1}\Big(\frac{1}{x(\ln x + 1)}\Big)$. 
Then the pattern $A^{-1}(\A{f}) = \int\!e^{\int\!\A{f}}$ gives automatically 
$k_0(x) = x \ln x$.\\
Therefore, an increasing scale on $(\frac{1}{e}, +\infty)$ 
is generated by the family
$$
k_\alpha = \begin{cases}
x \mapsto x^{\alpha x} &\textrm{if\ } \alpha \ne 0\,,\\
x \mapsto x \ln x&\textrm{if\ } \alpha = 0\,.
\end{cases}
$$
Moreover, it is now immediate to note that, in turn, the same family 
of functions generates a \textit{decreasing} scale on $(0,\tfrac{1}{e})$.

\vspace{2mm}
How about a possible reversing of Theorem~\ref{thm:mainresultL}\,? This point 
is rather fine; the existence of a scale implies a somehow weaker set of properties 
than the one assumed in Theorem~\ref{thm:mainresultL}. To the best of author's 
knowledge, the problem of finding a set of conditions {\it exactly\,} equivalent 
to generating a scale is still (and, most likely, widely) open. 

\begin{theorem}
\label{thm:mainresultR}
Let $U$ be an interval, $I = (a,\,b)$ an open interval, $(k_{\alpha})_{\alpha \in I}$, 
$k_\alpha \in \mathcal{C}^{2\ne}(U)$ for all \,$\alpha$. 

If $(\sr{k_{\alpha}})_{\alpha \in I}$ is an increasing scale then 
there exists an open dense subset $X \subset U$ such that the mappping
$I \ni \alpha \mapsto A(k_\alpha) (x) \in \R$ is increasing, 
1--1 and onto for all $x \in X$.
\end{theorem}

A proof of this theorem is given in Section~\ref{sec:MainResults}, 
immediately after the proof of Theorem~\ref{thm:mainresultL}.
\section{Properties and uses of $A$}
In what follows we will extensively use the operator $A$. 
Here we recall, after \cite{mikusinski}, some of its key properties. 
We also rephrase in the terms of $A$ an important result from~\cite{cargo}.

All this will be instrumental in showing that many nontrivial families of 
functions do generate scales. We will also deduce about the limit properties 
of our quasi-arithmetic means, stating a new result (Proposition~\ref{prop:kolesarova}) 
inspired, to some extent, by the paper \cite{kolesarova}.

Regarding scales as such, many examples of them were furnished in \cite[p.\,269]{bullen}.
Scales were also used by the old Italian school of statisticians; see, e.\,g., 
\cite{bonferroni192324,bonferroni192425,bonferroni1927,gini,pizzetti,ricci}. 
One of significant results from that last group of works will be presented, 
with a new and compact proof, in Proposition~\ref{prop:RadicalMeans}. That 
new approach will, we hope, show how quickly one can nowadays prove old 
results.

\begin{rem}\label{rem:increasing}
Let $U$ be an interval and $f,\,g \in \mathcal{C}^{2\ne}(U)$. 
Then the following conditions are equivalent:
\begin{enumerate}[\upshape (i)]
\item $A(f)(x) = A(g)(x)$ for all $x \in U$\,,
\item $f = \alpha g + \beta$ \,for some \,$\alpha,\,\beta \in \R$, 
\,$\alpha \ne 0$\,,
\item $\sr{f}(\underline{a},\,\underline{w})=\sr{g}(\underline{a},\,\underline{w})$
for all vectors $\underline{a} \in U^n$ and arbitrary corresponding weights 
$\underline{w}$ 
\end{enumerate} 
(see, for instance, \cite[p.\,66]{hlp}, \cite{mikusinski}).
\end{rem}

Let $f$ be a strictly monotone function such that $f \in \mathcal{C}^1(U)$ 
and $f'(x) \ne 0$ for all $x \in U$. Then there either holds $f'(x) < 0$ 
for all $x \in U$, or else $f'(x) > 0$ for all $x \in U$. So we define 
the sign $\mathrm{sgn}(f')$ of the first derivative of $f$ to be 
$\mathrm{sgn}(f')(x)$, where $x$ is any point in $U$. The key tool 
in our approach is 

\begin{proposition}[Basic comparison]\label{prop:basiccompare}
Let $U$ be an interval, $f,\,g \in \mathcal{C}^{2\ne}(U)$. 
Then the following conditions are equivalent: 
\begin{enumerate}[\upshape (i)]
\item $A(f) > A(g)$ on a dense set in $U$\,,
\item $(\mathrm{sgn} f') \cdot (f \circ g^{-1})$ is strictly convex\,,
\item $\sr{f}(\underline{a},\,\underline{w}) \ge \sr{g}(\underline{a},
\,\underline{w})$ for all vectors $\underline a \in U^n$ and weights 
$\underline w$, with both sides equal only when $\underline a$ is 
a constant vector.
\end{enumerate}
\end{proposition}
For the equivalence of (i) and (iii), see \cite[p.\,95]{mikusinski} (this 
characterization of comparability of means had, in the same time, been obtained
independently by S.\,{\L}ojasiewicz -- see footnote 2 in \cite{mikusinski}). 
For the equivalence of (ii) and (iii), see, for instance, \cite[p.\,1053]{CS64}. 
\vskip1mm
In the course of comparing means, one needs to majorate the difference between 
two means. If the interval $U$ is unbounded then, of course, the difference between 
any given two means can be unbounded (for example such is the difference between 
the arithmetic and geometric mean). In order to eliminate this drawback, we will 
henceforth suppose that the means are always defined on a compact interval. It will 
be with no loss of generality, because it is easy to check that a family of means 
defined on $U$ is a scale on $U$ if and only if those means form a scale on $D$,
when treated as functions $D \rightarrow \R$, for every closed subinterval 
$D \subset U$. Indeed, if $\underline a$ is a vector with values in $U$, then 
$\underline a$ is also a vector with values in $D$ for some closed subinterval 
$D$ of \,$U$. 

So, from now on, we have $U$ -- a compact interval, $g \in \mathcal{C}^{2\ne}(U)$ 
increasing, and $\A{g} \in L_1(U)$. The following theorem is of utmost technical 
importance. 

\begin{theorem}\label{thm:rightright}
Let $U$ be a closed bounded interval. If, for $n \in \mathbb{N}$, $\A{f} \colon 
U \rightarrow \R$, $\A{k_n} \in \mathcal{C}(U)$ and $\A{k_n}\xrightarrow[L_1]{}\A{f}$ 
then $\sr{k_n} \rightrightarrows \sr{f}$ uniformly with respect to $\underline{a}$ 
and $\underline{w}$. Moreover,
$$
\abs{\sr{f}(\underline{a},\underline{w})-\sr{k_n}(\underline{a},\underline{w})} 
\le \abs{U} e^{2\norma{\A f}_1} \sinh 2 \norma{\A{k}_n-\A{f}}_1
$$
for all $\underline{a}$ and $\underline{w}$ ($\norma{\cdot}_1$ 
is taken in the space $L_1(U)$).
\end{theorem}

\begin{proof}
Let $u = \inf U$. Solving a simple differential equation, in view of Remark 
\ref{rem:increasing}, it is possible to assume, for all considered functions $f$, 
that 
$$
f(x) = \int_{u}^x \exp \left(\int_{u}^s \A{f} (t) dt \right) ds,\quad x \in U.
$$
Much like in \cite[p.\,216]{cargo}, we have
$$\sr{f}(\underline{a},\underline{w})-\sr{k_n}(\underline{a},\underline{w})=
(f^{-1})'(\alpha) \sum_{1\le i \le j \leq m} 
p_ip_j\left(k_n(a_i)-k_n(a_j)\right)\left(\theta_n(z_i)-\theta_n(z_j)\right)$$
for certain $\alpha \in \left[ \min \underline{a} , \max \underline{a} \right]$, 
$\theta_n = (f \circ k_n^{-1})'$,\,\, $p_i \in (0,1)$,\,\, $\sum_{1 \le i \le j \le m} 
p_i p_j \le 1/4$. Now, continuing, 
\begin{align}
&\abs{\sr{f}(\underline{a},\underline{w})-\sr{k_n}(\underline{a},\underline{w})}\nonumber\\
&= \abs{(f^{-1})'(\alpha) \sum_{1\le i \le j \le m} 
p_ip_j(k_n(a_i)-k_n(a_j))\left(\theta_n(z_i)-\theta_n(z_j)\right)}\nonumber\\
&\le \norma{(f^{-1})'}_{\infty}\tfrac{1}{4}\left(k_n(\max\underline{a}) - k_n(\min\underline{a})\right) 
\,\,2\sup_{z,v \in U}\abs{\theta_n(z) - \theta_n(v)}\nonumber
\end{align}

Putting $\varepsilon := \norma{\A{k}_n - \A{f}}_{1}$, we assuredly have
$$
\frac{k_n'}{f'} = e^{\int\!\A{k}_n-\A{f}} \in (e^{-\varepsilon},\,e^\varepsilon)\,.
$$
So $\theta_n = (f\circ k_n^{-1})'(x) = \frac{f'\circ k_n^{-1}(x)}{k_n'\circ k_n^{-1}(x)} 
\in (e^{-\varepsilon},\,e^\varepsilon)$. What is more, 
\begin{align*}
k_n(\max \underline{a})-k_n(\min \underline{a}) 
&=\int_{\min \underline{a}}^{\max \underline{a}} k_n'(x)\,dx\\
&\le \int_{\min \underline{a}}^{\max \underline{a}} e^{\varepsilon}f'(x)\,dx\\
&= \,e^{\varepsilon} (f(\max\underline{a}) - f(\min\underline{a})). 
\end{align*}
Hence, continuing further, 
\begin{align*}
\abs{\sr{f}(\underline{a},\underline{w}) - \sr{k_n}(\underline{a},\underline{w})}
&\le  \frac{\norma{(f^{-1})'}_{\infty}}{2}(k_n(\max \underline{a}) - k_n(\min \underline{a})) 
\sup_{z,v \in U}\abs{\theta_n(z) - \theta_n(v)}\\
&\le \frac{\norma{(f^{-1})'}_{\infty} e^{\varepsilon}}{2}(f(\max\underline{a}) - f(\min\underline{a})) 
\abs{e^{\varepsilon} - e^{-\varepsilon}}\\
&\le \frac{\norma{f}_{\infty}}{\inf f'} e^{\varepsilon} \sinh \varepsilon\\
&\le \frac{\norma{f}_{\infty}}{\inf f'} \sinh 2\varepsilon\,. 
\end{align*}
But we also know that
\begin{equation}
\norma{f}_{\infty} = \norma{e^{\int\!\A f}}_{1} \le \abs{U}e^{\norma{\A f}_{1}} \label{eq:szac_1}
\end{equation}
and
\begin{equation}
\inf f'= \inf e^{\int\!\A f} \ge e^{-\norma{\A f}_{1}}. \label{eq:szac_2}
\end{equation}
So, prolonging the previous chain of estimations and using 
\eqref{eq:szac_1} and \eqref{eq:szac_2},
\begin{equation}
\abs{\sr{f}(\underline{a},\underline{w}) - \sr{k_n}(\underline{a},\underline{w})} 
\le \abs{U} e^{2\norma{\A f}_{1}} \sinh 2 \norma{\A{k}_n - \A{f}}_{1}.
\nonumber
\end{equation}
Hence $\sr{k_n} \rightrightarrows \sr{f}$. Theorem is now proved. 
\end{proof}

Heading towards the main results of the note, we state now 
\begin{proposition}\label{prop:tenascale}
Let $U$ be a closed bounded interval, $I = (a,\,b)$ -- an open interval, 
$(k_{\alpha})_{\alpha \in I}$ -- a family of functions from \,$\mathcal{C}^{2\ne}(U)$.

\noindent {\rm (A)} If $(\sr{k_{\alpha}})_{\alpha \in I}$ is an increasing 
scale then $(A(k_{\alpha}))_{\alpha \in I}$ satisfies all the conditions 
{\rm (a)} through {\rm (d)} listed below. 
\begin{enumerate}[\upshape (a)]
\item if $\alpha_i \rightarrow \alpha$ , then $A(k_{\alpha_i}) \rightarrow A(k_{\alpha})$,
\item if $\alpha < \beta$, then $A(k_{\alpha}) < A(k_{\beta})$ on a dense subset of \,$U$,
\item if $\alpha \rightarrow a+$, then $A(k_{\alpha})(x) \rightarrow -\infty$ on a dense subset of \,$U$,
\item if $\beta \rightarrow b-$, then $A(k_{\beta})(x) \rightarrow +\infty$ on a dense subset of \,$U$.
\end{enumerate}
\noindent {\rm (B)} Strengthening the pair of conditions {\rm (c)} and {\rm (d)} to 
\begin{enumerate}[\upshape (e)] 
\item ($\alpha \to a+ \Rightarrow A(k_{\alpha})(x) \rightarrow -\infty$) \quad and \quad 
($\beta \to b- \Rightarrow A(k_{\beta})(x) \rightarrow +\infty$)\qquad for all $x \in U$
\end{enumerate}
suffices to reverse the implication: {\rm (a)}, {\rm (b)} and {\rm (e)} implies $(\sr{k_{\alpha}})_{\alpha \in I}$ 
being an increasing scale. 
\end{proposition}

\begin{proof}
To simplify the notation, having $\underline a$ and $\underline w$ fixed, we write shortly
$$
F(\alpha) = \sr{k_{\alpha}}(\underline{a},\,\underline{w}),
$$
$F \colon I \rightarrow (\min\underline{a},\,\max\underline{a})$.
And then one simply checks step by step:
\begin{enumerate}[\upshape (a)]
\item if $\alpha_i \rightarrow \alpha+$ we have that $F(\alpha_i) \rightarrow F(\alpha)$. 

But it is easy to check that 
$$
\A{f}(x) \,= \,\lim_{\varepsilon\rightarrow 0+}\frac{2}{\varepsilon^2} 
\sr{f} (x - \varepsilon,\,x + \varepsilon)\,.
$$
So $\A{k_{\alpha_i}} \rightarrow \A{k_\alpha}$. 
\item if $\alpha < \beta$, we have $F(\alpha) \le F(\beta)$ 
and the equality holds iff $\underline a$ is constant. So by 
Proposition~\ref{prop:basiccompare} we have $\A{k_\alpha} < \A{k_\beta}$ on a dense set.\\
Let $E_{\alpha,\beta} = \{x \in U: \A{k_\alpha}(x) = \A{k_\beta}(x)\}$. 
We have that if $[\alpha ',\,\beta'] \subset [\alpha,\,\beta]$ then 
$E_{\alpha,\beta} \supset E_{\alpha',\,\beta'}$, and 
$E_{\alpha,\beta}$ is closed and nowhere dense. Thus 
$$
E\,\,=\bigcup_{{\alpha, \beta \in I} \atop {\alpha \ne \beta}} 
E_{\alpha,\beta}\,\,=\bigcup_{{\alpha, \beta \in I \cap \mathbb{Q}} \atop {\alpha \ne \beta}}
E_{\alpha,\beta}\,.
$$
So $E$ is closed and nowhere dense. Moreover, if $x \in U \,\backslash \,E$ 
and $\alpha < \beta$, we have $\A{k_\alpha}(x) < \A{k_\beta}(x)$. 
\item The proof is completely similar to that of (d) given below.
\item Let 
$$
K = \{x:\lim_{\beta \rightarrow b-} \A{k_{\beta}}(x) \not \rightarrow +\infty\}
$$
If $K$ is not a boundary set then there exist $c$, $d$, $c < d$, 
such that $[c,d] \subset K$. 
Let 
$$
M := \sup_{x \in [c,d]} \lim_{\beta \rightarrow b-} \A{k_{\beta}}(x)\,.
$$
Such a quantity $M$ is clearly finite. We have $\sr{k_\beta} (\underline v,\underline q) 
\le \sr{e^{Mx}} (\underline v,\underline q) < \max \underline v$ for all $\beta$ and 
$\underline v,\underline q$ such that $c \le \min \underline v \le \max \underline v \le d$. 
Hence the family $\{k_\beta\}$ does not generate a scale on $U$. So $K$ is dense. 
\end{enumerate}
\vskip1.2mm
To prove part (B) one needs to show that, under (e), $(\sr{k_\alpha})_{\alpha \in I}$ 
is a scale $U$. By Proposition~\ref{prop:basiccompare} we know that $F$ is 1--1. 
Additionally, when arguing to this side, we know that if $x \nearrow x_0$ then 
$\A{k_x} \nearrow \A{k_{x_0}}$. So $\A{k_x} \rightrightarrows \A{k_{x_0}}$ on 
$[\min\underline{a},\,\max\underline{a}]$. Therefore, by Theorem~\ref{thm:rightright}, 
we have $\sr{k_{x}} \rightrightarrows \sr{k_{x_0}}$ with respect to $\underline{a}$ 
and $\underline{w}$. Thus $F$ is continuous and 1--1.

To complete the proof, it is sufficient to show that
$$
\lim_{\alpha \rightarrow a+} F(\alpha) = \min{\underline a}\,,
\qquad 
\lim_{\beta \rightarrow b-} F(\beta) = \max{\underline a}\,.
$$
We know that $\A{k_\beta} \rightarrow +\infty$ on the closed interval $U$. 
So $\A{k_\beta} \rightrightarrows +\infty$ on $U$. Therefore, for any 
$M \in \R$ there exists $\beta_{M}$ such that 
$$
F(\beta) \ge \sr{e^M}(\underline{a},\underline{w})
$$
for all $\beta > \beta_{M}$. Now, taking $M \rightarrow +\infty$, and 
knowing that $\{e^{tx} \colon t \ne 0\} \cup \{x\}$ generates a scale 
on $\R$ (a folk-type theorem proved in \cite{burrows&talbot}; see 
also Remark~\ref{rem:etx}) we get
$$
F(\beta) \xrightarrow[\beta \rightarrow b-]{} \max \underline{a}\,.
$$ 
Similarly one may prove that
$$
F(\alpha) \xrightarrow[\alpha \rightarrow a+]{} \min \underline{a}\,.
$$ 
So $F$ is a continuous bijection between $I$ and $(\min\underline{a},\,\max\underline{a})$.
Hence $(\sr{k_\alpha})_{\alpha \in I}$ is a scale on $U$.
\end{proof}

\begin{rem}
\label{rem:etx}
To prove that the family $\{ e^{tx} \colon t \ne 0\} \cup \{x\}$ generates 
a scale on $\R$ it is enough, having data $\underline{a},\,\underline{w}$, 
to consider the all-positive-components-vector $\underline{v} = (e^{a_1},\dots,\,e^{a_n})$.
And then use the fact that the family of power means evaluated on $\underline{v}$ 
with weights $\underline{w}$ is a scale on $\R_{+}$.
\end{rem} 
\begin{corollary}[strenghtening of Proposition \ref{prop:tenascale}]\label{col:tenascale}
Let \,$U$ be an interval, $I = (a,b)$ -- an open interval, 
$(k_{\alpha})_{\alpha \in I}$, $k_\alpha \in \mathcal{C}^{2\ne}(U)$ 
for all \,$\alpha$. 

\noindent {\rm (A)} If $(\sr{k_{\alpha}})_{\alpha \in I}$ is an increasing 
scale then there exists an open dense set $X \subset U$ such that
\begin{enumerate}[\upshape (a)]
\item if $\alpha_i \rightarrow \alpha+$, then $A(k_{\alpha_i}) \rightarrow A(k_{\alpha})$
on $X$,
\item if $\alpha < \beta$ , then $A(k_{\alpha}) < A(k_{\beta})$ on $X$,
\item if $\alpha \rightarrow a+$, then $A(k_{\alpha})(x) \rightarrow -\infty$ on $X$, 
\item if $\beta \rightarrow b-$, then $A(k_{\beta})(x) \rightarrow +\infty$ on $X$.
\end{enumerate}
\noindent {\rm (B)} Under the stronger condition 
\begin{enumerate}[\upshape (e)] 
\item ($\alpha \to a+ \Rightarrow A(k_{\alpha})(x) \rightarrow -\infty$)\quad and\quad 
($\beta \to b- \Rightarrow A(k_{\beta})(x) \rightarrow +\infty$)\qquad for all $x \in U$
\end{enumerate}
the entire implication of the corollary can be reversed: 
{\rm (a)}, {\rm (b)} and {\rm (e)} implies that $(\sr{k_{\alpha}})_{\alpha \in I}$ 
is an increasing scale. 
\end{corollary}

This corollary says that in Proposition~\ref{prop:tenascale} one can have 
a single common subset (\,$X$\,) of $U$ on which conditions (a) through (d) hold.

\begin {proof}
We might assume that $U$ is a closed interval (compare the comment 
in the third paragraph below Proposition~\ref{prop:basiccompare}).

Let $E_{p,q} := \{x \in U \colon \, \A{k_p}(x) = \A{k_q}(x)\}$. 
Each $E_{p,q}$ is closed and nowhere dense, so 
$$
E := \{x \colon\exists_{p,q \in I} \,\,p \ne q \wedge \A{k_p}(x) = \A{k_q}(x)\}
$$
has two more descriptions 
$$
E \,= \bigcup_{{\alpha, \beta \in I} \atop {\alpha \ne \beta}} 
E_{\alpha,\beta} \,= \bigcup_{{\alpha, \beta \in I \cap \mathbb{Q}} \atop {\alpha \ne \beta}}
E_{\alpha,\beta}\,.
$$
We know that $E$ is closed nowhere dense, being a countable union of closed 
nowhere dense sets. So $X_{\ne} := U \,\backslash \,E$ is an open dense set. 
Let 
$$
X_{+\infty} := \{x \in U \colon \lim_{\beta \rightarrow b-} \A{k_{\beta}}(x) \rightarrow +\infty\}.
$$ 
By Proposition~\ref{prop:tenascale} we know that $X_{+\infty}$ is dense. 
We now prove that it is open. 
Let 
$$
X_{s} := \{x \in U \colon \lim_{\beta \rightarrow b-} \A{k_{\beta}}(x) > s\}.
$$
Observe that $X_s$ is dense (because $X_s\supset X_{+\infty}$). Moreover, 
for all $x_0 \in X_s$ there holds $\A{k_{\beta_0}}(x_0) > s + \delta$ for 
some $\beta_0 \in I$ and $\delta > 0$. Hence one may take an open neighborhood 
$P\ni x_0$ satisfying $\A{k_{\beta_0}}(x) > s + \tfrac{1}{2}\delta$ for all $x \in P$, 
implying $P \subset X_s$. So $X_s$ is open. But the mapping 
$\beta \mapsto \A{k_{\beta}}(x)$ is nondecreasing for all 
$x \in U$. Hence $X_{+\infty} = \bigcap_{n=1}^{\infty} X_n$ 
is open and dense. Similarly, 
$$
X_{-\infty} := \{x \in U \colon \lim_{\alpha \rightarrow a+} \A{k_{\alpha}}(x) 
\rightarrow -\infty\}
$$
is an open dense set as well. Now one may take 
$X:= X_{\ne} \cap X_{-\infty} \cap X_{+\infty}$. 

\noindent $X$ is clearly open and dense.
\end {proof}

\section{\label{sec:MainResults}Proofs of Theorems 1 and 2}

\begin{proof}[Proof of Theorem~\ref{thm:mainresultL}]
Let $U$ be an interval, $I = (a,b)$ an open interval, $X$ be a dense subset 
of $U$, where the mapping given in the wording of theorem is increasing and 
1--1. We work with the family of functions $(k_{\alpha})_{\alpha \in I}$, 
$k_\alpha \in \mathcal{C}^{2\ne}(U)$ for all \,$\alpha$. 

Let us take an arbitrary $x_0 \in X$. We know that 
$I \ni \alpha \mapsto \A{k_\alpha}(x_0)$ is increasing, 1--1 and onto $\R$.
Next, let us specify the function $\Phi:\R \rightarrow I$ such that 
$\A{k_{\Phi(\alpha)}}(x_0)=\alpha$. This function is increasing as well.

Then for $\alpha < \beta$ we have $\A{k_{\Phi(\alpha)}} < \A{k_{\Phi(\beta)}}$ 
on the dense subset of $U$ emerging from Corollary~\ref{col:tenascale}.
Due to the fact that $I \ni \alpha \mapsto \A{k_\alpha} (x) \in \R$ 
is onto, we have 
$$
\lim_{ \alpha \rightarrow a } \A{k_{\Phi(\alpha)}}(x) = -\infty
\, \qquad \textrm{ and }\qquad 
\lim_{ \beta \rightarrow b } \A{k_{\Phi(\beta)}}(x) = +\infty
$$
everywhere on $U$. So, using the part (B) of Corollary~\ref{col:tenascale}, 
the family of means $(\sr{k_{\alpha}})_{\alpha \in I}$ is an increasing 
scale on $U$.
\end{proof}

\begin{proof}[Proof of Theorem~\ref{thm:mainresultR}]
Let us take $X$ from Corollary~\ref{col:tenascale}. Let then fix any 
$x_0 \in X$. Let $\{s_p\}_{p \in \R}$ be the reparametrized family 
$\{k_\alpha\}_{\alpha \in I}$, with restriction 
$$
s_p = k_\alpha \textrm{ , where }p=\A{k_\alpha}(x_0).
$$ 
Then we know that the mapping
$$
\R \ni p \mapsto \A{s_p} (x) \in \R
$$
is 1--1 and onto for all $x \in X$, and if $p > q$
$$
\A{s_p}(x) > \A{s_q}(x).
$$
Moreover, due to the fact that $\A{s_p}(x_0)$ is onto, 
we have for all $x_0$ 
$$
\lim_{p \rightarrow -\infty} \A{s_p}(x_0) = -\infty\, \qquad \textrm{ and }
\qquad\lim_{p \rightarrow +\infty} \A{s_p}(x_0) = +\infty\,.
$$
So $p \mapsto \A{s_p}(x)$ is increasing, 1--1 and onto $\R$ for all $x \in X$.
\end {proof}

\section{\label{sec:applications} Applications}
\begin{proposition}[power means do generate a scale]\label{prop:PowerMeans}
Let \,$U = \,\R_{+}$ and \,$(r_{\alpha})_{\alpha \in \R}$, 
given by 
$$
r_\alpha(x) = \begin{cases} x^{\alpha} & \alpha \ne 0 \\ \ln x & \alpha = 0 \end{cases},
$$ 
be the family of power functions. Then the family $(r_\alpha)$ 
generates a scale on $\R_{+}$.
\end{proposition}

\begin{proof}
We compute $\A{r_\alpha}$, 
$$
\A{r_\alpha}(x) \,= \,\frac{\alpha - 1}{x}
$$
and see that the mapping $\alpha \mapsto \A{r_\alpha}(x)$ is 
increasing, 1--1 and onto for all $x \in \R_{+}$. So the assumptions 
in Theorem~\ref{thm:mainresultL} hold, implying that the family
$(r_\alpha)$ generates an increasing scale on $\R_+$.
\end{proof}

Before giving our second application we reproduce a 10 years old' result.
\begin{proposition}[\cite{kolesarova}]
Let $g \colon [0,1] \rightarrow \R$ be a continuous monotone function. 
Writing $g_\alpha(x) := g(x^{\alpha})$ for any $\alpha > 0$, there hold:
\begin{enumerate}[\upshape (i)]
\item if there exists the one side, nonzero derivative $g'(0+)$ then 
$$
\lim_{\alpha \rightarrow +\infty} \sr{g_\alpha} = \,\max\,,
$$
\item if there exists one side, nonzero derivative $g'(1-)$ then
$$
\lim_{\alpha \rightarrow 0+} \sr{g_\alpha} = \,\sr{\ln x}\,.
$$
\end{enumerate}
\end{proposition}
We prove a somehow similar, yet not so close, result.
\begin{proposition}
\label{prop:kolesarova}
Let $g \in \mathcal{C}^{2\ne} [0,1] \rightarrow (0,\,+\infty)$
(i.e., there exist the relevant one side second derivatives of $g$ 
at $0$ and $1$, too) and \,$g_{\alpha}(x) := g(x^{\alpha})$, 
$\alpha \in (0,\,+\infty)$. 
Then 
\begin{equation}
\lim_{\alpha \rightarrow 0+} \sr{g_\alpha} = \sr{\ln x}
\, \qquad \textrm{ and }\qquad
\lim_{\alpha \rightarrow +\infty} \sr{g_\alpha} = \max.
\label{eq:kolesarova}
\end{equation}
If, in addition, $g$ is convex,\footnote{in this situation we can just  
assume that $g \in \mathcal{C}^2[0,1)$, instead of assuming $g \in \mathcal{C}^{2\ne}[0,1)$,
because all convex, strictly monotone functions in $\mathcal{C}^2[0,1)$ belong to 
$\mathcal{C}^{2\ne}[0,1)$.} then $(g_\alpha)_{\alpha \in (0,\,+\infty)}$ 
generates a scale between the geometric mean and $\max$. 
\end{proposition}

\begin{proof}
We have to prove that the mapping
$(0,\,+\infty) \ni \alpha \mapsto \A{g_\alpha}(x) \in \R$ 
is 1--1 and onto for all $x \in (0,\,1)$. Let us fix an arbitrary 
$x \in (0,\,1)$. Then we have 
$$
\A{g_{\alpha}}(x) \,= \,\alpha x^{\alpha - 1}\A{g}(x^\alpha) + \frac{\alpha - 1}{x}\,.
$$
When $\alpha \rightarrow 0+$, then 
$$
\A{g_{0}}(x) := \lim_{\alpha \rightarrow 0+} \A{g_{\alpha}}(x) \,= \,\frac{-1}{x}\,.
$$
In turn, when $\alpha \rightarrow +\infty$, there holds
$$
\A{g_{\alpha}}(x) \,= \,\underbrace{\alpha x^{\alpha - 1}
\frac{g''(0)}{g'(0)}}_{>-\infty} + \,\frac{\alpha - 1}{x} \to +\infty\,.
$$
The proof of formulas \eqref{eq:kolesarova} is now completed. 

When, additionally, $g$ is convex, then $\A{g} \ge 0$ and, by 
Corollary~\ref{col:borderscale}, the family $\{g_{\alpha}\}_{\alpha \in \mathbb{R}_{+}}$ 
generates a scale on $(0,1)$ between the geometric mean and $\max$.
\end{proof}

Now we would like to present one classical result of the Italian School of 
statisticians from 1910-20s. That result has been reported in \cite[p.\,269]{bullen}. 
We now give it a new short proof based on Corollary~\ref{col:borderscale}.
\begin{proposition}[Radical Means]
\label{prop:RadicalMeans}
Let \,$U = \R_{+}$ and \,$(k_{\alpha})_{\alpha \in \R_{+}}$, 
\,$k_\alpha(x) = \alpha^{1/x}$, be the family of radical functions.
Then this family generates a decreasing scale on \,$\R_{+}$.
\end{proposition}

\begin{proof}
The proof appears to be quite close to the proof of 
Proposition~\ref{prop:PowerMeans}. Indeed, we quickly 
compute
$$
\A{k_\alpha}(x) = -\frac{2x + \ln \alpha}{x^2}\,,
$$
finding that the mapping $\alpha \mapsto \A{k_\alpha}(x)$ is decreasing, 
1--1 and onto for all $x \in \R_{+}$. So the assumption in 
Corollary~\ref{col:borderscale} hold, and hence the family 
$(k_\alpha)_{\alpha \in \R_+}$ generates a decreasing scale 
on $\R_{+}$.
\end{proof}

\textbf{Open problem.} How to unify Theorem~\ref{thm:mainresultL} 
and Theorem~\ref{thm:mainresultR} so as to get a set of conditions 
that would simultaneously be necessary and sufficient? 


\end{document}